\documentclass[12pt]{amsart}
\usepackage{amssymb}
\usepackage{amsmath}
\usepackage[margin=1in]{geometry}
\theoremstyle{plain}
\newtheorem{thm}{Theorem}[section]
\newtheorem*{main}{Main~Theorem}
\newtheorem*{Alex}{Aleksandrov's~Theorem}

\newtheorem{remark}[thm]{Remark}  
\newtheorem{claim}[thm]{Claim}
\newtheorem{corollary}[thm]{Corollary}
\newtheorem{lemma}[thm]{Lemma}
\usepackage{hyperref}
\hypersetup{%
	colorlinks=true,allcolors = blue,
} %%% New, allows for linking within the document
\usepackage[inline]{enumitem} %%% New, allows for greater control of itemized lists 
    \setenumerate[1]{label=$\bullet$}
\def\Xint#1{\mathchoice
{\XXint\displaystyle\textstyle{#1}}%
{\XXint\textstyle\scriptstyle{#1}}%
{\XXint\scriptstyle\scriptscriptstyle{#1}}%
{\XXint\scriptscriptstyle\scriptscriptstyle{#1}}%
\!\int}
\def\XXint#1#2#3{{\setbox0=\hbox{$#1{#2#3}{\int}$ }
\vcenter{\hbox{$#2#3$ }}\kern-.6\wd0}}

\def\dashint{\Xint-}

\makeatletter %%% New, adds acknowledgements
\newcommand\ackname{Acknowledgements}
\if@titlepage
  \newenvironment{acknowledgements}{%
      \titlepage
      \null\vfil
      \@beginparpenalty\@lowpenalty
      \begin{center}%
        \bfseries \ackname
        \@endparpenalty\@M
      \end{center}}%
     {\par\vfil\null\endtitlepage}
\else
  \newenvironment{acknowledgements}{%
      \if@twocolumn
        \section*{\abstractname}%
      \else
        \small
        \begin{center}%
          {\bfseries \ackname\vspace{-.5em}\vspace{\z@}}%
        \end{center}%
        \quotation
      \fi}
      {\if@twocolumn\else\endquotation\fi}
\fi
\makeatother

\renewcommand{\div}{\operatorname{div}}
\newcommand{\trace}{\operatorname{tr}}
\newcommand{\ess}{\operatorname{ess}}

\newcommand{\ip}[2]{\ensuremath{\langle #1 , #2 \rangle}}
\newcommand{\tp}{\ensuremath{\textnormal{\texttt{p}}}}
\newcommand{\tq}{\ensuremath{\textnormal{\texttt{q}}}}

\theoremstyle{definition}
\newtheorem{definition}{Definition}

\theoremstyle{remark}

\numberwithin{subcase}{case}

\numberwithin{equation}{section}
\begin{document}
\title[Equivalence of solutions to the $\tp(\cdot)$-Laplacian in $\mathbb{R}^n$]{A New Proof of the Equivalence of Weak and Viscosity Solutions to the Homogeneous $\tp(\cdot)$-Laplacian in $\mathbb{R}^n$}

\author{Zachary Forrest}
\email{zachary9@usf.edu}
\address{Department of Mathematics and Statistics,
University of South Florida, Tampa, FL 33620, USA}

\author{Robert D. Freeman}
\email{rdf5256@psu.edu}
\address{Department of Mathematics, Pennsylvania State University Berks, Reading, PA 19610, USA}

\thanks{This paper is part of the first author's Ph.D. thesis.}

\maketitle

\begin{abstract}
We present a new proof for the equivalence of potential theoretic weak solutions and viscosity solutions to the $\tp(\cdot)$-Laplace equation in $\mathbb{R}^n$.  The proof of the equivalence in the variable exponent case in Euclidean space was  first given by Juutinen, Lukkari, and Parviainen (2010) and extended the equivalence of potential theoretic weak solutions and viscosity solutions to the $\tp$-Laplace equation in $\mathbb{R}^n$, given by Juutinen, Lindqvist, and Manfredi (2001). In both the fixed exponent case and the variable exponent case, the main argument is based on the maximum principle for semicontinuous functions, several approximations, and also applied the full uniqueness machinery of the theory of viscosity solutions. This paper extends the approach of Julin and Juutinen (2012) for the fixed exponent case, and so we employ infimal convolutions to present a direct, new proof for the equivalence of potential theoretic weak solutions and viscosity solutions to the $\tp(\cdot)$-Laplace equation in $\mathbb{R}^n$. 
\end{abstract}

\begin{acknowledgements}
 The authors would like to thank Dr. Thomas Bieske for his numerous discussions and patient guidance; Dr. Diego Riccioti for his insights on regularity; and Dr. David Cruz-Uribe, whose interest and encouragement at the 2022 Ohio River Analysis Meeting has led us to publish our results.
\end{acknowledgements}

\tableofcontents

\section{Introduction and Motivation}\label{intro}
The model case for variational problems and partial differential equations with non-standard growth conditions is the $\tp(\cdot)$-Laplace equation, which is the generalization of the constant $\tp$-Laplace case. The reasons for studying this variable exponent case are numerous; modern image restoration techniques are modeled by the $\tp(\cdot)$-Laplacian (see, for example, \cite{CLR}), as is fluid flow of eltrorhealogical fluids (refer to \cite{D} and \cite{R} for examples). The purpose of this paper is to present a new proof of the equivalence of \textit{potential theoretical weak solutions} and \textit{viscosity solutions} to the homogeneous $\tp(\cdot)$-Laplace equation (see Section \ref{solutions} for definitions). 

The equivalence of potential theoretic weak solutions and viscosity solutions to the homogeneous $\tp$-Laplace equation was shown by Juutinen, Mandfredi, and Lindqvist (see \cite{JLM}).  In \cite{JLP}, Juutinen,  Lukkari, and Parviainen extended the results of \cite{JLM} to show the equivalence of potential theoretic weak solutions and viscosity solutions to the homogeneous $\tp(\cdot)$-Laplace equation. (See Section \ref{solutions} for definitions.) In both the fixed exponent and variable exponent cases, the authors exploited the maximum principle for semicontinuous functions and required, as the authors state in the variable exponent case, “delicate” approximations. Additionally, both cases needed the full uniqueness machinery of the theory of viscosity solutions.  This, perhaps, led to the motivation for a new, more direct proof of the equivalence of potential theoretic weak solutions and viscosity solutions to the $\tp$-Laplace equation, as presented in \cite{JJ}, which utilizes \textit{infimal convolutions} (see Section \ref{infconrev}).  
Theorem \ref{main2} of this paper offers a new proof to the equivalence of potential theoretic weak supersolutions and viscosity supersolutions to the $\tp(\cdot)$-Laplace equation, different than \cite{JLP} and in the spirit of \cite{JJ}. 

The setup is as follows. Let $\Omega \subset \mathbb{R}^n$ be an open and bounded set and $1<\tp(\cdot)<\infty$ where $\tp(\cdot) \in \textnormal{C}^1(\Omega)$.
The $\tp(\cdot)$-Laplacian is defined as
\begin{eqnarray}\label{p(x)laplace1}
-\Delta_{\tp(\cdot)}u(x) := -\div(|D u(x)|^{\tp(\cdot)-2}D u(x))=0, 
\end{eqnarray}
where $x$ is a point in $\Omega$, $u: \Omega \to \mathbb{R}$ possesses appropriate regularity, and $D u$ denotes the gradient of $u$.  Note that Equation \eqref{p(x)laplace1} is the Euler-Lagrange equation of the functional
\begin{eqnarray*}
u \mapsto \int_{\Omega} \frac{1}{\tp(\cdot)} |D u(x)|^{\tp(\cdot)} dx.
\end{eqnarray*}
By taking $\tp(x) \equiv \tp \in (1,\infty)$ we recover the constant exponent operator which is defined by
\begin{eqnarray*}
-\Delta_{\tp} u(x):=-\div(|D u(x)|^{\tp-2}D u(x)) 
\end{eqnarray*}
and which is the Euler-Lagrange equation of the functional
\begin{eqnarray*}
u \mapsto \frac{1}{\tp} \int_{\Omega}  |D u(x)|^{\tp} dx. %,
\end{eqnarray*}
There are two obvious differences in the variable exponent case which are worth bringing to the reader's attention: 
First is that as the exponent involves a function at the point $x \in \Omega$, the behavior of the integrand is necessarily more complicated than the standard, fixed exponent case; and the other is that the fraction $1/\tp(\cdot)$ naturally appears in the integrand and cannot be factored out in the variable exponent case. This presents some difficulty in comparison to the constant exponent case, especially since differentiation will force the use of logarithms.

Since the $\tp(\cdot)$-Laplace equation is degenerate when $\tp(\cdot) > 2$ and singular for $1<\tp(\cdot)<2$, then in order to guarantee that the Dirichlet boundary value problem is solvable, typical approaches involve utilizing the distributuional weak solutions. 

We will prove Theorem \ref{main2} (see p. 4) by showing that viscosity supersolutions to the $\tp(\cdot)$-Laplace equation are $\tp(\cdot)$-superharmonic functions; a simple application of the comparison principle for weak solutions implies that $\tp(\cdot)$-superharmonic functions are viscosity supersolutions to the $\tp(\cdot)$-Laplace equation. 
(See Section \ref{solutions} for definitions.) 

An analogous argument shows that weak subsolutions and viscosity subsolutions are of the same class of solution and thus we have the equivalence of potential theoretic weak solutions and viscosity solutions to the $\tp(\cdot)$-Laplace equation.

Key to the proof are the properties of the infimal convolutions: By their definition and Aleksandrov's Theorem for convex functions, we are able to approximate semicontinuous functions by functions which are $\textnormal{C}^2$ a.e. in $\Omega$. This permits us to avoid certain estimates necessary under the methods of \cite{JLP} which, in turn, result in inconvenient restrictions in some nonstandard spaces. (See, for instance, \cite{BF}.)

As in \cite{JJ}, we begin with the identity
\begin{eqnarray}\label{IBPstartup}
\int_{\Omega}\big(-\Delta_{\tp(\cdot)}u(x)\big) \psi dx = \int_{\Omega}|Du|^{\tp(\cdot)-2}Du \cdot D\psi dx,
\end{eqnarray}
which is true provided $u$ and $\psi$ are of sufficient regularity and $\psi$ is compactly supported on $\Omega$. The proof is then split into the equation's degenerate and singular cases. 

In the degenerate case, we use the infimal convolution mentioned above to produce an increasing sequence $(u_\varepsilon)_{\varepsilon>0}$ of semiconcave functions satisfying 
$-\Delta_{\tp(\cdot)}u_\varepsilon(\cdot) \geq 0$ in a useful subset $\Omega_{r(\varepsilon)}$ of $\Omega$ (defined in Section \ref{infconrev}).  
Approximations utilizing the properties of $u_\varepsilon$ allow us to conclude
\begin{eqnarray}\label{NTSoverall}
\int_{\Omega} |D u_\varepsilon|^{\tp(\cdot)-2} \ip{D u_\varepsilon}{D \psi} dx \geq  \int_{\Omega} \left(-\Delta_{\tp(\cdot)}u_\varepsilon(x)\right) \psi dx \geq 0;
\end{eqnarray}
passing this inequality to the limit as $\varepsilon \to 0$ proves that $u$ is $\tp(\cdot)$-superharmonic. 

The singular case is similar but requires care since it is not clear what happens when $Du = 0$ in the \textit{nondivergence} form (see Section \ref{solutions}). This obstacle forces us to regularize Equation \eqref{IBPstartup} in order to perform a similar argument and calculations as in the singular case. (See Case 2 in the proof of Theorem \ref{main2}.) The result of these cases is Theorem \ref{main2}:

\begin{thm}\label{main2}
Weak supersolutions and viscosity supersolutions to the $\textnormal{\tp}(\cdot)$-Laplace equation coincide.
\end{thm}

\noindent
By duplicating the arguments for Theorem \ref{main2} for $\tp(\cdot)$-weak subharmonic and viscosity $\tp(\cdot)$-subsolutions, the Main Theorem follows as a corollary:

\begin{main}
Weak solutions and viscosity solutions to the $\textnormal{\tp}(\cdot)$-Laplace equation coincide.
\end{main} 
It should be noted that a recent article by Medina and Ochoa has appeared on the math archive which shows the new proof of the equivalence of potential theoretic weak supersolutions and viscosity supersolutions to the nonhomogeneous case of the $\textnormal{\tp}(\cdot)$-Laplace equation.  Their proof uses infimal convolutions and is similar to the homogeneous case presented in this paper; however, the authors of this paper have derived this work independently with no prior knowledge of the work of Medina and Ochoa. The authors wish to recognize Medina and Ochoa's achievement, and present the contents of the current article as an additional result in the field.

The paper is divided up in the following way:  in Section \ref{soborvw}, we recall some key properties of variable exponent Lebesgue and Sobolev spaces.  We discuss notions of solutions to the $\tp(\cdot)$-Laplace equation and some relevant facts relating some of these notions in Section \ref{solutions}.  In Section \ref{infconrev} we recall some key results involving infimal convolutions and recall Aleksandrov's Theorem. In Section \ref{newproof} we present the new proof of the equivalence of potential theoretic weak supersolutions and viscosity supersolutions to the $\textnormal{\tp}(\cdot)$-Laplace equation (Theorem \ref{main2}).  Then we state the Main Theorem, which is a consequence of Theorem \ref{main2}.

\section{The spaces $L^{\tp(\cdot)}(\Omega)$ and $W^{1,\tp(\cdot)}(\Omega)$}\label{soborvw} 
In this section, we review some key properties of variable exponent Lebesgue spaces and Sobolev spaces employing the variable exponent $\tp(\cdot)$.  (See \cite{CU},\cite{KR} for a more complete discussion.)
We let $\tp: \mathbb{R}^n \rightarrow (1,\infty)$, called a variable exponent, be in $\textnormal{C}^1(\mathbb{R}^n)$ and let $\Omega \subset \mathbb{R}^n$ be a bounded, open set.

Denote the following
$$\tp^+ = \sup_{x \in \Omega} \tp(\cdot) \textmd{\ \ and \ \ } \tp^- = \inf_{x \in \Omega} \tp(\cdot)$$
and assume that
$$1 < \tp^- \le \tp^+ < \infty $$
holds in compact subsets of $\Omega$.
\subsection{Variable Exponent Lebesgue Spaces}
We first define the variable exponent Lebesgue space as in \cite{KR}:  $L^{\tp(\cdot)}(\Omega)$ is the space of measurable functions $u$ on $\Omega \subset \mathbb{R}^n$ such that the modular $\varrho_{\tp(\cdot)}$ satisfies 
$$\varrho_{\tp(\cdot)}(u)\: := \: \int_{\Omega} | u(x)|^{\tp(\cdot)} dx < \infty.$$ 
Moreover, we use the Luxemburg norm:
\begin{equation*} 
\| u \|_{L^{\tp(\cdot)}(\Omega)} : =  \inf \bigg{\{} \; \lambda > 0\; : \; \int_{\Omega} \bigg|\frac{u(x)}{\lambda}\bigg|^{\tp(\cdot)} dx \leq 1 \;\; \bigg{\}}.
\end{equation*}
Note that because $\tp^+<\infty$,  $L^{\tp(\cdot)}(\Omega)$ equipped with this norm is a Banach space.  Also note that  if ${\tp(\cdot)}$ is constant, then $L^{\tp(\cdot)}(\Omega)$ reduces to the standard Lebesgue space $L^\tp(\Omega)$.

The definition of the norm produces the following relationship between the modular and the norm:
\begin{eqnarray}\label{norm2mod}
\min\{{\|u\|_{L^{\tp(\cdot)}(\Omega)}^{\tp^+},\|u\|_{L^{\tp(\cdot)}(\Omega)}^{\tp^-}}\} \; \le \; \varrho_{\tp(\cdot)}(u) \:\:\le\:\:
\max\{{\|u\|_{L^{\tp(\cdot)}(\Omega)}^{\tp^+},\|u\|_{L^{\tp(\cdot)}(\Omega)}^{\tp^-}}\}.
\end{eqnarray}

These inequalities \eqref{norm2mod} directly imply for any sequence $\{u_k\}_{k\in\mathbb{N}}\stackrel{k \rightarrow \infty}{\longrightarrow}u$, we have:
\begin{eqnarray}\label{equivnorm}
\varrho_{\tp(\cdot)}(u-u_k) \rightarrow 0  \:\: \Longleftrightarrow \:\: \|u-u_k\|_{L^{\tp(\cdot)}(\Omega)} \rightarrow 0
\end{eqnarray}
as $k \rightarrow \infty$.

Given functions $u \in L^{\tp(\cdot)}(\Omega)$ and $v \in L^{\texttt{q}(\cdot)}(\Omega)$, where the  conjugate exponent $\texttt{q}(\cdot)$ of $\tp(\cdot)$ is defined pointwise, we have a form of H\"{o}lder's inequality (cf. \cite[Theorem 2.1]{KR}):
\begin{eqnarray}\label{Holder}
\int_\mathcal{O} |u| \: |v| \; dx \leq \: C \: \|u\|_{L^{\tp(\cdot)}(\Omega)} \|v\|_{L^{\texttt{q}(\cdot)}(\Omega)}.
\end{eqnarray}
Additionally, the dual of $L^{\tp(\cdot)}(\Omega)$ is $L^{\texttt{q}(\cdot)}(\Omega)$ and $L^{\tp(\cdot)}(\Omega)$ is reflexive. 

\subsection{Variable Exponent Sobolev Spaces}
We will use the following notation and definition for the variable exponent  Sobolev space $W^{1,\tp(\cdot)}(\mathbb{R}^n)$ for $1 < \tp^+ < \infty$:

\begin{eqnarray*}
W^{1,\tp(\cdot)}(\mathbb{R}^n) =  \bigg{\{f} \in L^{\tp(\cdot)}(\mathbb{R}^n), |D f| \in L^{\tp(\cdot)}(\mathbb{R}^n)\; : \; \int_{\mathbb{R}^n} |f(x)|^{\tp(\cdot)} + |D f(x)|^{\tp(\cdot)} dx < \infty \bigg{\}},
\end{eqnarray*}
where we use the norm
\begin{equation*}
\| f \|_{W^{1,\tp(\cdot)}(\mathbb{R}^n)}  = \| f \|_{L^{\tp(\cdot)}(\mathbb{R}^n)}  + \| D f\|_{L^{\tp(\cdot)}(\mathbb{R}^n)},
\end{equation*}
which makes $W^{1,\tp(\cdot)}(\mathbb{R}^n)$ a Banach space (\cite[Theorem 3.4]{HHP}).  
Similarly, we define the variable exponent  Sobolev space $W^{1,\tp(\cdot)}(\Omega)$ for $1<\tp^+ < \infty$ in the natural way, where we let $\Omega \subset \mathbb{R}^n$ be an open and connected set in $\mathbb{R}^n$.  

Replacing $ L^{\tp(\cdot)}(\Omega)$ by $L_{\textnormal{loc}}^{\tp(\cdot)}(\Omega)$, we define the space $ W_{\textnormal{loc}}^{1,{\tp(\cdot)}}(\Omega)$, which consists of functions $f$
that belong to $ W_{\textnormal{loc}}^{1,{\tp(\cdot)}}(\Omega')$ for all open sets 
$\Omega' \Subset \Omega$, in the natural way.  The space $W_{0}^{1,{\tp(\cdot)}}(\Omega)$ is the closure in $W^{1,{\tp(\cdot)}}(\Omega)$ of smooth functions with compact support. Additionally, it is a Banach space and it is reflexive.

We note that the density of smooth functions in $W^{1,\tp(\cdot)}(\Omega)$ is typically a nontrivial matter; %but is not an issue since we are 
by assuming that $\tp(\cdot)$ is a $\textnormal{C}^1$ function, we avoid the issue. Indeed, since $\tp(\cdot) < \infty$ we have that continuous functions with compact support are dense in $L^{\tp(\cdot)}(\Omega)$. (See, for example, \cite[Theorem 3.3]{HHP2}.) In addition, we have $\textnormal{C}^{\infty}_0(\Omega)$ is dense in $W^{1,\tp(\cdot)}(\Omega)$ \cite[Theorem 3]{S}.
This density allows us to pass from smooth test functions to Sobolev test functions in the definition of weak solutions given in Section \ref{solutions}.

\section{Notions of Solutions to the $\tp(\cdot)$-Laplace Equation}\label{solutions}
In this section, we will discuss various notions of solutions to the $\tp(\cdot)$-Laplace equation, along with some key results we will need to prove Theorem \ref{main2}.  
Recall the $\tp(\cdot)$-Laplace Equation introduced in Section \ref{intro}:
\begin{eqnarray}\label{p(x)laplace}
-\Delta_{\tp(\cdot)}u(x) := -\div(|D u(x)|^{\tp(\cdot)-2}D u(x)) = 0.
\end{eqnarray}
We may also formally compute the divergence to obtain the nondivergence form of the $\tp(\cdot)$-Laplace operator:   
\begin{eqnarray}\label{nondivergent}
-\Delta_{{\tp(\cdot)}}u & = & -\div(| D u |^{\tp(\cdot)-2}D u)   \nonumber
\\ & = &- | D u |^{\tp(\cdot)-2} \trace \left(D^2 u\right) -(\tp(\cdot)-2) | D u |^{\tp(\cdot)-4} \ip{D^2u D u}{D u} \nonumber
\\ & & \mbox{}- \log(| D u |)| D u |^{\tp(\cdot)-2} \ip{ D \tp(\cdot)}{D u} \nonumber
\\ & = & -|D u|^{\tp(\cdot) -2}\big(\Delta u + (\tp(\cdot) -2)  \ip{D^2u \frac{D u}{|D u|}}{\frac{D u}{|D u|}} \big) - |D u|^{\tp(\cdot) -2} \log|D u| \ip{D \tp(\cdot)}{D u}
\\ & = & -|D u|^{\tp(\cdot) -2}\big(\Delta u + (\tp(\cdot) -2) \Delta_{\infty} u \big) - |D u|^{\tp(\cdot) -2} \log|D u| \ip{D \tp(\cdot)}{D u}, \nonumber
\end{eqnarray}
where $D^2u$ is the Hessian matrix, $\Delta u$ is the usual Laplacian (which is the trace of $D^2u$), and $\Delta_{\infty} u := \ip{D^2u \frac{D u}{|D u|}}{\frac{D u}{|D u|}}$ is the normalized $\infty$-Laplacian.

We first define the notion of weak solutions to Equation \eqref{p(x)laplace}.
\begin{definition}
The function $u \in W_{\textnormal{loc}}^{1,\tp(\cdot)}({\Omega})$ is a \emph{weak supersolution} to Equation \eqref{p(x)laplace} if 
$$\int_\Omega | D u |^{\tp(\cdot)-2} \ip{D u}{D \phi} dx \;\; \geq \;  0$$
for all nonnegative test functions $\phi \in C_0^{\infty}({\Omega})$.  
The function $u \in W_{\textnormal{loc}}^{1,\tp(\cdot)}({\Omega})$ is a \emph{weak subsolution} to Equation \eqref{p(x)laplace} if $-u$ is a weak supersolution to Equation \eqref{p(x)laplace}. 
The function $u \in W_{\textnormal{loc}}^{1,\tp(\cdot)}({\Omega})$ is a \emph{weak solution} to Equation \eqref{p(x)laplace} if it is both a weak supersolution and a weak subsolution.
\end{definition}

\begin{remark}\label{testfunc}
If $u \in W^{1,\textnormal{\tp}(\cdot)}({\Omega})$, we may use test functions in $W_0^{1,\textnormal{\tp}(\cdot)}({\Omega})$ via standard approximation arguments.
Moreover, by regularity theory, we can identify weak solutions with their locally H\"{o}lder continuous representative. (See \cite{AM}, \cite{Yu} and \cite{FZ}.)
\end{remark}

We next consider the following class of $\tp(\cdot)$-superharmonic functions and $\tp(\cdot)$-subharmonic functions.
\begin{definition}\label{superharm}
The function $u: \Omega \to \mathbb{R} \cup \{\infty\}$ is \emph{$\textnormal{\tp}(\cdot)$-superharmonic} if 
\begin{enumerate}
\item[(\ref{superharm}$\textnormal{a}$)] $u$ is lower semicontinuous,
\item[(\ref{superharm}$\textnormal{b}$)] $u$ is finite almost everywhere, and
\item[(\ref{superharm}$\textnormal{c}$)] the comparison principle holds: If $v$ is a weak solution to Equation \eqref{p(x)laplace} in some domain $D \Subset \Omega$, $v$ is continuous in $\overline{D}$, and $u \geq v$ on $\partial D$, then $u\geq v$ in $D$.
\end{enumerate} 
The function $u: \Omega \to \mathbb{R} \cup \{-\infty\}$ is \emph{$\textnormal{\tp}(\cdot)$-subharmonic} if $-u$ is $\tp(\cdot)$-superharmonic. 
The function $u: \Omega \to \mathbb{R} \cup \{\infty\}$ is \emph{$\textnormal{\tp}(\cdot)$-harmonic} if it is both  $\tp(\cdot)$-subharmonic and  $\tp(\cdot)$-superharmonic.
\end{definition}

\begin{remark}
The second condition in the definition of $\textnormal{\tp}(\cdot)$-superharmonic (respectively $\textnormal{\tp}(\cdot)$-subharmonic) requires that $u$ is finite almost everywhere, whereas the classic definition usually states that $u$ is not identically $+\infty$ $(-\infty)$ in each component of $\Omega$.  The stronger condition stated here is needed for the characterization of $\textnormal{\tp}(\cdot)$-superharmonic (subharmonic) functions as pointwise increasing limits of supersolutions (subsolutions) to Equation \eqref{p(x)laplace}.  For a more complete discussion, see \cite[Section 6]{HHKLM}.
\end{remark}

We now review some key facts in the variable exponent case necessary to show the equivalence of potential theoretic weak solutions and $\tp(\cdot)$-harmonic functions.  We will use the lower semicontinuous regularization given by
\begin{eqnarray}\label{lscreg}
u_{\star}(x) := \ess \lim \inf_{y \to x}u(y) = \lim_{r \to 0} \ess \inf_{B(x,r)} u.
\end{eqnarray}
The next two theorems imply that a weak solution to Equation \eqref{p(x)laplace} is equivalent to also being a $\tp(\cdot)$-harmonic solution, which we state as Corollary \ref{weakharm} for reference.
Theorem \ref{weaksuperh} follows from \cite[Theorem 6.1]{HHKLM} and \cite[Theorem 4.1]{HKL} and Theorem \ref{superhweak} follows from \cite[Corollary 6.6]{HHKLM}.

\begin{thm}\label{weaksuperh}
Let $1<\textnormal{\tp}(\cdot)<\infty$ and assume $u$ is a  weak supersolution in $\Omega$.  Then $u=u_{\star}$ a.e. and $u_{\star}$ is $\textnormal{\tp}(\cdot)$-superharmonic.
\end{thm}

\begin{thm}\label{superhweak}
Let $1<\textnormal{\tp}(\cdot)<\infty$ and assume $u$ is a locally bounded $\textnormal{\tp}(\cdot)$-superharmonic function.  Then $u$ is a weak supersolution.
\end{thm}

\begin{corollary}\label{weakharm}
Weak supersolutions and $\textnormal{\tp}(\cdot)$-superharmonic solutions to Equation \eqref{p(x)laplace} coincide.
\end{corollary}

Recall the $\tp(\cdot)$-Laplace equation in non-divergence form:
\begin{equation}\label{non}
    \begin{aligned}
        -\bigg( \: \|D u\|^{\tp(\cdot)-2} \trace(D^2u)  +  (\tp(\cdot)-2) \|D u\|^{\tp(\cdot)-4}\ip{(D^2u)D u}{D u}& \\
        +\|D u\|^{\tp(\cdot)-2}\log(\|D u\|)\ip{D \tp(\cdot)}{D u }\bigg) &= 0
    \end{aligned}
%\lefteqn{-\bigg( \: \|D u\|^{\tp(\cdot)-2} \trace(D^2u)  +  (\tp(\cdot)-2) \|D u\|^{\tp(\cdot)-4}
%\ip{(D^2u)D u}{D u} } && \\
%&&\mbox{} +\|D u\|^{\tp(\cdot)-2}\log(\|D u\|)\ip{D \tp(\cdot)}{D u }\bigg) =  0 \nonumber
\end{equation}
in a bounded domain $\Omega$.
We note the left hand side of Equation \eqref{non} is degenerate elliptic and proper in the sense of \cite{CIL:UGTVS}.
We will also need the following definitions.
Given the function $u$, we consider the following definition for the set of functions $\Upsilon$ that touch $u$ from below at the point $x_0$.  That is, the set $\mathcal{TB}(u,x_0)$ given by  $$\mathcal{TB}(u,x_0)=\{\Upsilon \in \textnormal{C}^2(\Omega) :u(x_0)=\Upsilon(x_0),u(x) > \Upsilon(x)\; \textmd{for} \; x\neq x_0, \; x \; \textnormal{near} \; x_0 \}.$$  Similarly, the set of functions $\Upsilon$ that touch $u$ from above at the point $x_0$ is given by $$\mathcal{TA}(u,x_0)=\{\Upsilon \in \textnormal{C}^2(\Omega) :u(x_0)=\Upsilon(x_0),u(x) < \Upsilon(x)\; \textmd{for} \; x\neq x_0, \; x \; \textnormal{near} \; x_0 \}.$$

We are now able to define the concepts of \emph{viscosity solutions} to Equation \eqref{non}:
\begin{definition}\label{visc}
The function $u:\Omega \to \mathbb{R} \cup \{\infty\}$ is a \emph{viscosity supersolution} to Equation \eqref{non} if the following hold:
\begin{enumerate}
\item[($\ref{visc}\textnormal{a}$)] $u$ is lower semicontinuous,
\item[($\ref{visc}\textnormal{b}$)] $u$ is finite almost everywhere, and
\item[($\ref{visc}\textnormal{c}$)] For $x_0 \in \Omega$,  $\phi \in  \mathcal{TB}(u,x_0)$, with $D \phi(x_0) \ne 0$, satisfies 
$$-\Delta_{\tp(\cdot)} \phi(x_0)  \geq 0.$$
\end{enumerate} 
A function $u$ is a \emph{viscosity subsolution}  to Equation \eqref{non} if
$-u$ is a viscosity supersolution.  
That is, the function $u:\Omega \to \mathbb{R} \cup \{-\infty\}$ is a \emph{viscosity subsolution} to Equation \eqref{non} if the following hold:
\begin{enumerate}
\item[(\ref{visc}$\textnormal{d}$)] $u$ is upper semicontinuous,
\item[(\ref{visc}$\textnormal{e}$)] $u$ is finite almost everywhere, and
\item[(\ref{visc}$\textnormal{f}$)] For $x_0 \in \Omega$,  $\psi \in  \mathcal{TA}(u,x_0)$, with $D  \psi(x_0) \ne 0$, satisfies 
$$-\Delta_{\tp(\cdot)} \psi(x_0)  \leq 0.$$
\end{enumerate} 
A function $u$ is a \emph{viscosity solution} if it is both a viscosity supersolution and a viscosity subsolution.
\end{definition}

\section{Background Results}\label{infconrev}
In this section, we review the definition of infimal convolutions, some useful notations, mention Aleksandrov's Theorem and some key results involving infimal convolutions which we will need in the proof of Theorem \ref{main2}.
\subsection{Infimal convolutions and notations}\label{infimalconvolutions}
As in \cite{JJ}, the main tool we will use to show viscosity supersolutions are weak supersolutions to Equation \eqref{p(x)laplace} is infimal convolutions.  
Recall that $\tp(\cdot) \in \textnormal{C}^1(\Omega)$ and we use the notations:
$$\tp^- := \inf_{x \in \Omega} \tp(\cdot) \;\; \textnormal{and} \;\; \tp^+ := \sup_{x \in \Omega} \tp(\cdot).$$  
We will split the proof of Theorem \ref{main2} in Section \ref{newproof} into two cases:  the degenerate case and the singular case.  Each case will need a different infimal convolution.
\begin{itemize}
\item[Case 1:] \; The infimal convolution we will use for the proof when $2 \leq \tp^- < \infty$ is 
\begin{eqnarray}\label{infconv}
u_{\varepsilon}(x) := \inf_{y \in \Omega} \bigg(u(y) + \frac{|x-y|^{\tq}}{ \tq \varepsilon^{\texttt{q}}}\bigg),
\end{eqnarray}
where $2\leq \tq < \infty$ and $\varepsilon > 0$.  
\item[Case 2:] \; In the case that $1 < \tp^+ < 2$, we will use the infimal convolution
\begin{eqnarray}\label{infconv2}
u_{\varepsilon}(x) := \inf_{y \in \Omega} \bigg(u(y) + \frac{|x-y|^{\tq}}{ \textnormal{\texttt{q}} \varepsilon^{\textnormal{\texttt{q}}-1}}\bigg),
\end{eqnarray}
where $\tq > \frac{\tp^+}{\tp^+ - 1}$, %so
implying $\tq > 2$.
\end{itemize}

Case 1 corresponds to the degenerate case defined in Section \ref{intro} and Case 2 corresponds to the singular case. We will use the following notation and definition of $\Omega_{r(\varepsilon)}$ in both cases: 
$$\Omega_{r(\varepsilon)}:= \{x \in \Omega \; : \; dist(x,\partial\Omega)> r(\varepsilon)\}.$$ 

\begin{definition}\label{Yepsilon}
Let $u: \Omega \to \mathbb{R}$ and $x \in \Omega_{r(\varepsilon)}$.  We define the set $Y_{\varepsilon}$ as in \cite{JJ}.  That is, we have the following condition for membership in the set $Y_{\varepsilon}$: 
\begin{eqnarray}\label{Yeps}
y_0 \in Y_{\varepsilon} \;\;\; \textnormal{if} \;\;\; u_{\varepsilon}(x) = u(y_0) + \frac{1}{\tq \varepsilon^{\tq -1}}|x-y_0|^{\tq},
\end{eqnarray}
or in other words, for some $x \in \Omega_{r(\varepsilon)}$, we have 
$$u(y_0) + \frac{1}{\tq \varepsilon^{\tq -1}}|x-y_0|^{\tq} \; = \; \inf_{y \in \Omega} \left(u(y) + \frac{|x-y|^{\tq}}{\textnormal{\texttt{q}}\varepsilon^{\textnormal{\texttt{q}}-1}}\right).$$
\end{definition}
Note that $Y_{\varepsilon}$ is nonempty and closed for every $x \in \Omega_{r(\varepsilon)}$ since $u$ is lower semicontinuous and bounded.

\subsection{Some key results}
In this section, we present three key results involving infimal convolutions and state Aleksandrov's Theorem.  
(These can be found within the literature involving infimal convolutions and convex functions. See, for example \cite{M}, \cite{ST}.)
The proof of the first lemma is identical to the fixed exponent case presented in \cite{JJ}. For the sake of completeness, we include the proof below. 

\begin{lemma}\label{infconvprop1}
Suppose $u:\Omega \to \mathbb{R}$ is bounded and lower semicontinuous in $\Omega$. Assume $1< \textnormal{\tp}(\cdot)<\infty$.
\begin{itemize} 
\item[(i)] There exists $r(\varepsilon) > 0$ such that
$$u_{\varepsilon}(x) := \inf_{y \in B_{r(\varepsilon)}(x) \cap\Omega}\bigg(u(y) + \frac{|x-y|^{\textnormal{\texttt{q}}}}{\textnormal{\texttt{q}}\varepsilon^{\tq-1}}\bigg).$$
Furthermore, we have $r(\varepsilon) \to 0$ as $\varepsilon \to 0$.
\item[(ii)] The sequence $(u_{\varepsilon})$ is increasing and $(u_{\varepsilon}) \to u$ pointwise in $\Omega$.
\item[(iii)] If $u$ is a viscosity supersolution to $-\Delta_{\textnormal{\tp}(\cdot)} u \geq 0$ in $\Omega$, then the infimal convolution $u_{\varepsilon}$ is a viscosity solution to $-\Delta_{\textnormal{\tp}(\cdot)} u_{\varepsilon} \geq 0$ in $\Omega_{r(\varepsilon)}$.
\end{itemize}
\end{lemma}
\begin{proof}
Notice $(ii)$ follows from $(i)$ and the definition of $u_{\varepsilon}$.  We show $(i)$ and $(iii)$.
\begin{itemize}
\item[(i):] Choose $r(\varepsilon)$ so that
\begin{eqnarray*}
\frac{1}{\tq} \frac{\big(r(\varepsilon)\big)^{\tq}}{\varepsilon^{\tq-1}} = \sup_{\Omega}u - \inf_{\Omega}u=:M - m.
\end{eqnarray*}
Then $r(\varepsilon) \to 0$ as $\varepsilon \to 0$ and, fixing $x \in \Omega_{r(\varepsilon)}$, for all $y \in \Omega \setminus \overline{B}_{r(\varepsilon)}(x)$ we have
\begin{eqnarray*}
u(y) + \frac{|x-y|^{\tq}}{ \texttt{q} \varepsilon^{\tq-1}} > m + \frac{\big(r(\varepsilon)\big)^{\tq}}{\tq\varepsilon^{\tq-1}} \geq M \geq u(x).
\end{eqnarray*} 
\item[(iii):] By $(i)$, for every $x \in \Omega_{r(\varepsilon)}$ we have 
\begin{eqnarray*}
u_{\varepsilon}(x) = \inf_{y \in B_{r(\varepsilon)}(x)}\bigg(u(y) + \frac{|x-y|^{\tq}}{\tq\varepsilon^{\tq-1}}\bigg) = \inf_{z \in B_{r(\varepsilon)}(0)}\bigg(u(z-x) + \frac{|z|^{\tq}}{\tq\varepsilon^{\tq-1}}\bigg).
\end{eqnarray*} 
We need to show that for every $z \in B_{r(\varepsilon)}(0)$, the function $$\phi_{z}(x)=u(z-x) + \frac{|z|^{\tq}}{\tq\varepsilon^{\tq-1}}$$ is a viscosity %supersolution
solution to $-\Delta_{\tp(\cdot)}\phi_{z}(x) \geq 0$ in $\Omega_{r(\varepsilon)}$.  
We have
\begin{eqnarray*}
-\Delta_{\tp(\cdot)}\phi_{z}(x) = -\Delta_{\tp(\cdot)} \big(u(z-x) + \frac{|z|^{\tq}}{\tq\varepsilon^{\tq-1}}\big) = -\Delta_{\tp(\cdot)}u(z-x).
\end{eqnarray*}
Now, since $u$ is a viscosity supersolution to 
the $\tp(\cdot)$-Laplace equation
%$-\Delta_{\tp(\cdot)}u \geq 0$ 
in $\Omega_{r(\varepsilon)}$, then $\phi_{z}(x)$ is a viscosity supersolution to 
the $\tp(\cdot)$-Laplace equation
%$-\Delta_{\tp(\cdot)}\phi_{z}(x) \geq 0$ 
in $\Omega_{r(\varepsilon)}$.  
Since $u_{\varepsilon}$ is an infimum of $\phi_{z}$,
and since the operator $-\Delta_{\tp(\cdot)}$ is continuous, it must be that $u_\varepsilon$ is a viscosity supersolution to 
the $\tp(\cdot)$-Laplace equation
%$-\Delta_{\tp(\cdot)} u_{\varepsilon} \geq 0$ 
in $\Omega_{r(\varepsilon)}$.\qedhere
\end{itemize}
\end{proof}
We have another lemma, whose proof we omit.
\begin{lemma}[Lemma A.2, \cite{JJ}]\label{infconvprop2}
Suppose that $u: \Omega \to \mathbb{R}$ is bounded and lower semicontinuous.  Then $u_{\varepsilon}$ is semi-concave in $\Omega_{r(\varepsilon)}$.  In other words, there is a constant $C$ such that 
$$ x \mapsto u_{\varepsilon}(x) - C|x|^2 $$ is concave where $C$ depends only upon $\tq$, $\varepsilon$, and $\sup_{\Omega}u - \inf_{\Omega}u$.
\end{lemma}

From Lemma \ref{infconvprop2}, it follows that $u_{\varepsilon}$ is locally Lipschitz and from Aleksandrov's Theorem (see below) $u_{\varepsilon}$ is twice differentiable almost everywhere in $\Omega$.
Additionally, we can write
\begin{eqnarray}\label{secderivest}
D^2 u_{\varepsilon}(x) \leq 2CI_n \; \; \; \; \textnormal{for a.e.}\; \;  x \in \Omega_{r(\varepsilon)}.
\end{eqnarray}

We have another lemma:
\begin{lemma}[Lemma 4.3, \cite{JJ}]\label{lscuscgradeps0}
Suppose $u: \Omega \to \mathbb{R}$ is bounded and lower semicontinuous.  Let $u_{\varepsilon}$ be the infimal convolution of $u$ defined by Equation \eqref{infconv2}.  Then
\begin{enumerate}[label=\textbf{(\Alph*)}]
    \item\label{wutbA} The function $x \mapsto \max_{y \in Y_{\varepsilon}(x)}|y-x|$ is upper semicontinuous.

    \item\label{wutbB} Assume $D u_{\varepsilon}(x)$ exists.  Then for every $y \in Y_{\varepsilon}(x)$,
    $$\left(\frac{|x-y|}{\varepsilon}\right)^{\tq -1} \leq |D u_{\varepsilon}(x)|.$$
Specifically, if $D u_{\varepsilon}(x) = 0$, then $u_{\varepsilon}(x) = u(x)$.
\end{enumerate}
\end{lemma}

We conclude this section with Aleksandrov's Theorem.  For a proof of Aleksandrov's Theorem, see, for example, \cite[Theorem 1, Section 6.3]{EG}.
\begin{Alex}
Let $u: \mathbb{R}^n \to \mathbb{R}$ be convex.  Then $u$ has a second derivative $n$-dimensional Lebesgue measurable a.e. on $\mathbb{R}^n$. More precisely, for $n$-dimensional Lebesgue measure a.e. $x$, 
$$\Big| u(y) - u(x) - Du(x)\cdot (y-x) - \frac{1}{2}(y-x)^T \cdot D^2u(x) \cdot (y-x)) \Big| = o(|y-x|^2)$$
as $y \to x$.
\end{Alex}
\begin{remark}\label{alexremark}
Owed to the proof of Aleksandrov's Theorem and the discussion of convexity in \cite[Theorem 1, Section 6.3]{EG}, we know that an $n$-dimensional Lebesgue measurable a.e. point $x$ satisfies the following condtions:
\begin{itemize}
\item[(1)] $Du(x)$ exists and $\lim_{r\to 0} \dashint_{B(x,r)}|Du(y)-Du(x)|dy=0$.
\item[(2)] $\lim_{r\to 0} \dashint_{B(x,r)}|D^2u(y)-D^2u(x)|dy=0$.
\item[(3)] $\lim_{r\to 0} |[D^2u]_s| \frac{B(x,r)}{r^n}=0$, where $[D^2u]_s$ is the singular part of the derivative $Du$.
\end{itemize}
\end{remark}

\section{The New Proof}\label{newproof}
We are now ready to prove Theorem \ref{main2}, which asserts that weak supersolutions and viscosity supersolutions to the $\textnormal{\tp}(\cdot)$-Laplace equation coincide. As mentioned in Section \ref{intro}, the proof may be split into singular and the degenerate cases; however, it should be noted that the proofs of both cases share techniques. We mention these commonalities here to improve reading of the subsections which follow -- details will be shown in Subsection \ref{degcase}.

First, it is easily shown that a $\tp(\cdot)$-superharmonic function is a $\tp(\cdot)$-viscosity supersolution (see, for example, \cite{JLM}); hence we focus our attention on the proof that $\tp(\cdot)$-viscosity supersolutions are $\tp(\cdot)$-superharmonic. Second, we will regularize our $\tp(\cdot)$-viscosity supersolution $u:\Omega \to \mathbb{R}^n$ by replacing $u$ by its infimal convolution $u_{\varepsilon}$ -- the details will vary slightly between cases. This will permit us to conclude that $(u_{\varepsilon})_{\varepsilon>0}$ is an increasing sequence of solutions to the $\tp(\cdot)$-Laplace equation
in
\begin{equation*}
	\Omega_{r(\varepsilon)}:= \{x \in \Omega \; : \; dist(x,\partial\Omega)> r(\varepsilon)\}
\end{equation*}
with the property that $\varphi:=u_\varepsilon-C|\cdot|^2$ is concave in $\Omega_{r(\varepsilon)}$. By showing that the operator $-\Delta_{{\tp(\cdot)}}$ or its regularization by $\delta>0$ (utilized in Subsection \ref{sincase}) possesses a lower bound in $\Omega_{r(\varepsilon)}$, we produce a sequence of smooth functions $(u_{\varepsilon,j})$ converging to $u$ as $j \to \infty,\varepsilon \to 0$ and pass to Fatou's Lemma. In both cases, letting $\delta \to 0$ in the second, we will have established our desired inequality:
\begin{equation*}
	\int_\Omega |Du|^{\tp(\cdot)-2}\ip{Du}{D\psi}dx \geq 0, \;\; \psi \in C_0^\infty(\Omega) \,\, \text{and} \;\,\, \psi \geq 0.
\end{equation*}
  
\subsection[p->=2]{The Degenerate Case}\label{degcase}
Recalling the definition of $\tp^-$, we establish the following restriction of Theorem \ref{main2}.

\begin{thm}\label{main2case1}
	Weak supersolutions and viscosity supersolutions to the $\textnormal{\tp}(\cdot)$-Laplace equation coincide when $\tp^- \geq 2$.
\end{thm}

\begin{proof}
Assume $\tp^- \geq 2$ and take the infimal convolution as in Equation \eqref{infconv}, but with $\tq =2$.  
Specifically, we use:
$$u_{\varepsilon}(x) := \inf_{y \in \Omega} \bigg(u(y) + \frac{|x-y|^{2}}{ 2 \varepsilon^{2}}\bigg),$$
as in \cite{JJ}.  
Let $k \in \mathbb{R}$ and assume that $u$ is a viscosity supersolution to the $\textnormal{\tp}(\cdot)$-Laplace equation. 
Then $\min\{u,k\}$ is a viscosity supersolution to the $\textnormal{\tp}(\cdot)$-Laplace equation.  Since we can replace $u$ with $\min\{u,k\}$, we may assume that $u$ is locally bounded.  
Note that if we have the case that $\min\{u,k\}$ is $\tp(\cdot)$-superharmonic, then we must have that $u$ is $\tp(\cdot)$-superharmonic by definition.

By Lemma \ref{infconvprop1}, $(u_{\varepsilon})_{\varepsilon>0}$ is an increasing sequence of semi-concave viscosity supersolutions to $-\Delta_{\tp(\cdot)}u_{\varepsilon} =0$ in 
$\Omega_{r(\varepsilon)}$ which converges to $u$.  
In particular, we have
$$\varphi(x) := u_{\varepsilon}(x) - C |x|^2$$
is concave in $\Omega_{r(\varepsilon)}$.  Note that we may take 
$C=(2\varepsilon)^{-1}$.
Since $|x|^2,\varphi$ are concave, then Aleksandrov's Theorem implies $u_{\varepsilon}$ is twice differentiable almost everywhere and therefore
\begin{eqnarray}\label{nonepsilon}
	-\Delta_{\tp(\cdot)}u_{\varepsilon}(x) & = & - | D u_{\varepsilon}(x) |^{\tp(\cdot)-2} \trace \left(D^2 u_{\varepsilon}(x)\right) \nonumber
	\\ &&\mbox{}-(\tp(\cdot)-2) | D u_{\varepsilon}(x) |^{\tp(\cdot)-4} \ip{D^2u_{\varepsilon}(x) D u_{\varepsilon}(x)}{D u_{\varepsilon}(x)} 
	\\ & & \mbox{}- \log(| D u_{\varepsilon}(x) |)| D u_{\varepsilon}(x) |^{\tp(\cdot)-2} \ip{ D \tp(\cdot)}{D u_{\varepsilon}(x)} \nonumber
	\\ & \geq & 0 \nonumber
\end{eqnarray}
almost everywhere in $\Omega_{r(\varepsilon)}$.
We note that $D^2 u_{\varepsilon}(x)$ is the Hessian matrix of $u_{\varepsilon}$ at $x$ in the sense of Alexandrov.  
Now, Equation \eqref{nonepsilon} and \cite[Theorem 2.3]{L} imply that we only need to show that
\begin{eqnarray}\label{NTS}
	\int_{\Omega} |D u_{\varepsilon}|^{\tp(\cdot)-2} \ip{D u_{\varepsilon}}{D \psi} dx \geq \int_{\Omega} \left(-\Delta_{\tp(\cdot)}u_{\varepsilon}(x)\right) \psi dx
\end{eqnarray}
for any non-negative $\psi \in C_{0}^{\infty}\left(\Omega_{r(\varepsilon)}\right) \subseteq C_{0}^{\infty}\left(\Omega\right)$.  
We fix such a function $\psi$ and by standard mollification we can produce $(\varphi_j)$, 
a sequence of smooth concave functions converging to $\varphi$. Now let 
$$u_{\varepsilon,\; j}(x) := \varphi_j(x) + \frac{1}{2\varepsilon}|x|^2.$$

\begin{claim}\label{jconv}
	$u_{\varepsilon,\; j} \to u_{\varepsilon}$ as $j \to \infty$.
\end{claim}
\begin{proof}[Proof of Claim \ref{jconv}]
\renewcommand{\qedsymbol}{$\blacksquare$}	
	Observe that $\varphi_j \to \varphi$ as $j \to \infty$.  Then 
	\begin{eqnarray*}
		\lim_{j \to \infty} u_{\varepsilon,\; j}(x) & = & \lim_{j \to \infty} \left(\varphi_j(x) + \frac{1}{2\varepsilon}|x|^2 \right) \\
		& = & \varphi(x) + \frac{1}{2\varepsilon}|x|^2 \\
		& = & u_{\varepsilon}(x)-C|x|^2 + \frac{1}{2\varepsilon}|x|^2.
	\end{eqnarray*}
	The choice $C:=(2\varepsilon)^{-1}$ yields
	\begin{equation*}
		\lim_{j \to \infty} u_{\varepsilon,\; j}(x) =  u_{\varepsilon}(x)-\frac{1}{2\varepsilon}|x|^2 + \frac{1}{2\varepsilon}|x|^2 = u_{\varepsilon}(x).\qedhere
	\end{equation*}
\end{proof}
Applying integration by parts to the smooth functions $u_{\varepsilon,\; j}$ and recalling $\psi$ is compactly supported in $\Omega$ gives us
\begin{eqnarray}\label{IBPimplies}
	\int_{\Omega} |D u_{\varepsilon,\; j}|^{\tp(\cdot)-2} \ip{D u_{\varepsilon,\; j}}{D \psi} dx = \int_{\Omega} \left(-\Delta_{\tp(\cdot)}u_{\varepsilon,\; j}(x)\right) \psi dx.
\end{eqnarray}
Now since $u_{\varepsilon}$ is locally Lipschitz continuous, Claim \ref{jconv} implies
\begin{eqnarray}\label{limweakepsj}
	\lim_{j \to \infty} \int_{\Omega} |D u_{\varepsilon,\; j}|^{\tp(\cdot)-2} \ip{D u_{\varepsilon,\; j}}{D \psi} dx =\int_{\Omega} |D u_{\varepsilon}|^{\tp(\cdot)-2} \ip{D u_{\varepsilon}}{D \psi} dx.
\end{eqnarray}

\begin{claim}\label{hessepsj}
	$D^2 u_{\varepsilon,\; j}(x) \leq \frac{1}{\varepsilon} I_n$, where $I_n$ represents the identity matrix.
\end{claim}
\begin{proof}[Proof of Claim \ref{hessepsj}]
\renewcommand{\qedsymbol}{$\blacksquare$}	
	Observe that
	\begin{eqnarray*}
		D^2 u_{\varepsilon,\; j} = D^2\left(\varphi_j + \frac{1}{2\varepsilon}|x|^2 \right) = D^2 \varphi_j + D^2 \left(\frac{1}{2\varepsilon}|x|^2\right) = D^2 \varphi_j + \frac{1}{\varepsilon} I_n.
	\end{eqnarray*}
	Recall that the derivative of a concave function must be strictly monotonically decreasing and the Hessian $D^2 \varphi_j$ must be negative semi-definite. Then we must have 
	\begin{equation*}
		D^2 u_{\varepsilon,\; j}(x) \leq \frac{1}{\varepsilon} I_n.\qedhere
	\end{equation*}
\end{proof}
Observe that both $D u_{\varepsilon,\; j}(x)$ and $D \tp(\cdot)$ are locally bounded so that we can write 
$$|D u_{\varepsilon,\; j}(x)| \leq \tilde{C} \;\; \textnormal{and} \;\; |D \tp(\cdot)| \leq \kappa,$$
where $\tilde{C} < \infty$ and $\kappa < \infty$ and we can choose both  
$\tilde{C},\kappa>1$.
\begin{claim}\label{lapvarepsj}
	Since $D^2 u_{\varepsilon,\; j}(x) \leq \frac{1}{\varepsilon} I_n$ and both $D u_{\varepsilon,\; j}(x)$ and $D \textnormal{\tp}(\cdot)$ are locally bounded, we must have 
	\begin{eqnarray}\label{new}
		-\Delta_{\textnormal{\tp}(\cdot)} u_{\varepsilon,\; j}(x) \geq  - \frac{\tilde{C}^{\textnormal{\tp}^+ -2}n\left[\textnormal{\tp}^+ + \tau -1\right]}{\varepsilon}
	\end{eqnarray}
	in the support of $\psi$ where $\tau := \varepsilon \log(|\tilde{C}|) \kappa \tilde{C}$. 
\end{claim}
\begin{proof}[Proof of Claim \ref{lapvarepsj}]
\renewcommand{\qedsymbol}{$\blacksquare$}	
	We estimate:
	\begin{eqnarray*}
		-\Delta_{\tp(\cdot)} u_{\varepsilon,\; j}(x) &=&  - | D u_{\varepsilon,\; j}(x) |^{\tp(\cdot)-2} \trace D^2 u_{\varepsilon,\; j}(x) \\
		&&\mbox{}-(\tp(\cdot)-2) | D u_{\varepsilon,\; j}(x) |^{\tp(\cdot)-4} \ip{D^2u_{\varepsilon,\; j}(x) D u_{\varepsilon,\; j}(x)}{D u_{\varepsilon,\; j}(x)} \\
		&& \mbox{}- \log(| D u_{\varepsilon,\; j}(x) |)| D u_{\varepsilon,\; j}(x) |^{\tp(\cdot)-2} \ip{ D \tp(\cdot)}{D u_{\varepsilon,\; j}(x)} \\
		& \geq &- \tilde{C}^{\tp^+ -2}\frac{1}{\varepsilon}n - \left(\tp(\cdot) -2\right) \tilde{C}^{\tp^+ -4}\ip{\frac{1}{\varepsilon}I_n\tilde{C}\langle1,1,...,1 \rangle}{\tilde{C}\langle1,1,...,1 \rangle} \\
		&&\mbox{}- \log{\left(|\tilde{C}|\right)} \tilde{C}^{\tp^+ -2} \ip{D \tp(\cdot)}{\tilde{C} \langle1,1,...,1 \rangle} \\
		%-\Delta_{\tp(\cdot)} u_{\varepsilon,\; j}(x) 
		& \geq & -\tilde{C}^{\tp^+ -2}\frac{1}{\varepsilon}n - \left(\tp^+ -2\right) \tilde{C}^{\tp^+ -4}\left(\frac{\tilde{C}^2 n}{\varepsilon}\right) 
		\\ && \mbox{}- \frac{\varepsilon}{\varepsilon} \tilde{C}^{\tp^+ -2}\log(|\tilde{C}|)\ip{\kappa\langle1,1,...,1 \rangle}{\tilde{C}\langle1,1,...,1 \rangle} \\
		& \geq & -\tilde{C}^{\tp^+ -2}\frac{1}{\varepsilon}n - \left(\tp^+ -2\right) \tilde{C}^{\tp^+ -4}\left(\frac{\tilde{C}^2 n}{\varepsilon}\right) - \frac{1}{\varepsilon} \tilde{C}^{\tp^+-2} n \tau  \\
		& = & - \frac{\tilde{C}^{\tp^+ -2}n\left[1 + \left(\tp^+ -2\right) + \tau \right]}{\varepsilon} \\
		& = & - \frac{\tilde{C}^{\tp^+ -2}n\left[\tp^+ + \tau -1\right]}{\varepsilon} \\
	\end{eqnarray*}
	%Since $\log{\left(|\tilde{C}|\right)}$ is constant, there exists a constant $D u_{\varepsilon}$ such that we can write 
	%$$\log{\left(|\tilde{C}|\right)} := \frac{D u_{\varepsilon,j}}{\varepsilon}.$$
	%It follows that 
	%\begin{eqnarray*}
	%	-\Delta_{\tp(\cdot)} u_{\varepsilon,\; j}(x) & \geq & -\tilde{C}^{\tp(\cdot) -2}\frac{1}{\varepsilon}n - \left(\tp^- -2\right) \tilde{C}^{\tp(\cdot) -4}\left(\frac{\tilde{C}^2 n}{\varepsilon}\right) - \frac{D u_{\varepsilon}}{\varepsilon}  \tilde{C}^{\tp(\cdot) -2} \tilde{C}n \kappa \\
	%	& \geq & -\tilde{C}^{\tp(\cdot) -2}\frac{1}{\varepsilon}n - \left(\tp^- -2\right) \tilde{C}^{\tp(\cdot) -4}\left(\frac{\tilde{C}^2 n}{\varepsilon}\right) - \frac{1}{\varepsilon}  \tilde{C}^{\tp(\cdot) -2} n \tau  \\
	%	& = & - \frac{\tilde{C}^{\tp(\cdot) -2}n\left[1 + \left(\tp^- -2\right) + \tau \right]}{\varepsilon} \\
	%	& = & - \frac{\tilde{C}^{\tp(\cdot) -2}n\left[\tp^- + \tau -1\right]}{\varepsilon} \\
	%\end{eqnarray*}
	and the claim is proved.
\end{proof}
Next, employing Fatou's Lemma we have
\begin{eqnarray}\label{fatou}
	\int_{\Omega} \liminf_{j \to \infty} \left(-\Delta_{\tp(\cdot)} u_{\varepsilon,\; j}\right)\psi dx \leq \liminf_{j \to \infty} \int_{\Omega} \left(-\Delta_{\tp(\cdot)} u_{\varepsilon,\; j}\right)\psi dx 
\end{eqnarray}
Since $D^2 \varphi_j(x) \to D^2 \varphi(x)$ for almost every $x$ \cite[Sections 6.2 and 6.3]{EG}, then
\begin{eqnarray}\label{epsjtoeps}
	\liminf_{j \to \infty} \left(-\Delta_{\tp(\cdot)} u_{\varepsilon,\; j}(x)\right) = -\Delta_{\tp(\cdot)} u_{\varepsilon}(x)
\end{eqnarray}
for almost every $x$, and we have the desired result.  Indeed, 
\begin{eqnarray*}
	\lim_{j \to \infty} \int_{\Omega} |D u_{\varepsilon,\; j}|^{\tp(\cdot)-2} \ip{D u_{\varepsilon,\; j}}{D \psi} dx & = & \int_{\Omega} |D u_{\varepsilon}|^{\tp(\cdot)-2} \ip{D u_{\varepsilon}}{D \psi} dx \\
	& = & \lim_{j \to \infty} \int_{\Omega} \left(-\Delta_{\tp(\cdot)}u_{\varepsilon,\; j}(x)\right) \psi dx
\end{eqnarray*} 
by Equation \eqref{limweakepsj} and Equation \eqref{IBPimplies} respectively.  Then by Equation \eqref{fatou} and Equation \eqref{epsjtoeps}, respectively, we have
$$\liminf_{j \to \infty} \int_{\Omega} \left(-\Delta_{\tp(\cdot)}u_{\varepsilon,\; j}(x)\right) \psi dx \; \; \geq \;\; \int_{\Omega} \liminf_{j \to \infty} \left(-\Delta_{\tp(\cdot)} u_{\varepsilon,\; j}\right)\psi dx = \int_{\Omega} \left(-\Delta_{\tp(\cdot)}u_{\varepsilon}(x)\right) \psi dx.$$
It follows that we now have Equation \eqref{NTS}, or
\begin{equation*}
	\int_{\Omega} |D u_{\varepsilon}|^{\tp(\cdot)-2} \ip{D u_{\varepsilon}}{D \psi} dx \geq \int_{\Omega} \left(-\Delta_{\tp(\cdot)}u_{\varepsilon}(x)\right) \psi dx.\qedhere
\end{equation*}
\end{proof}

\subsection[1<p+<2]{The Singular Case}\label{sincase}
We now turn our attention to the singular case -- that is, the case $1<\tp^+<2$. The proof presents obstacles that we did not encounter in Section \ref{degcase}.  Here, as mentioned in the introduction, it is not clear what is meant by
\begin{eqnarray*}
	-\Delta_{{\tp(\cdot)}}u & = & -|D u|^{\tp(\cdot) -2}\left(\Delta u + (\tp(\cdot) -2)  \ip{D^2u \frac{D u}{|D u|}}{\frac{D u}{|D u|}} \right) \\
	& \mbox{}\hspace{2.5ex} & - |D u|^{\tp(\cdot) -2} \log|D u| \ip{D \tp(\cdot)}{D u} 
\end{eqnarray*}
when $D u = 0$.  To overcome this, we regularize and apply the identity
\begin{eqnarray*}
	\int_{\Omega}\left(-\div \left(|D u|^2 + \delta \right)^{\frac{\tp(\cdot)-2}{2}} D u\right) \psi dx
	= \int_{\Omega}\left(|D u|^2 + \delta \right)^{\frac{\tp(\cdot)-2}{2}} \ip{D u}{D \psi} dx
\end{eqnarray*} 
and attempt to pass to the limit when we let $\delta \to 0$. This will also permit us to identify integrable lower bounds and apply Fatou's Lemma as in the previous result.

\begin{thm}\label{main2case2}
	Weak supersolutions and viscosity supersolutions to the $\textnormal{\tp}(\cdot)$-Laplace equation coincide when $1<\tp^+<2$.
\end{thm}
\begin{proof}
	We assume $1< \tp^+ < 2$ and will use the infimal convolution as in Equation \eqref{infconv2}.  That is, we will use 
	\begin{eqnarray*}
		u_{\varepsilon}(x) := \inf_{y \in \Omega} \bigg(u(y) + \frac{|x-y|^{\tq}}{ \textnormal{\texttt{q}} \varepsilon^{\textnormal{\texttt{q}}-1}}\bigg),
	\end{eqnarray*}
	where $\tq > \frac{\tp^-}{\tp^- - 1}$. Note that $\tq >2$.
	Aleksandrov's Theorem and Claims \ref{jconv} and \ref{hessepsj} apply as before and show that $-\Delta_{\tp(\cdot)} u_{\varepsilon} \geq 0$ a.e. in $\Omega_{r(\varepsilon)}$ and
	\begin{eqnarray}\label{NTS2}
		\int_{\Omega} \left(|D u_{\varepsilon}|^2 + \delta \right)^{\frac{\tp(\cdot)-2}{2}}\ip{D u_{\varepsilon}}{D \psi} dx \geq \int_{\Omega} -\div \left(\left( D u_{\varepsilon}|^2 + \delta \right)^{\frac{\tp(\cdot)-2}{2}} D u_{\varepsilon} \right) \psi dx
	\end{eqnarray}
		for $\psi \in C_0^{\infty}(\Omega)$, $\psi \geq 0$ and for any $\delta > 0$.  We want to let $\delta \to 0$ and obtain
	\begin{eqnarray*}
		\int_{\Omega} |D u_{\varepsilon}|^{\tp(\cdot) - 2}\ip{D u_{\varepsilon}}{D \psi} dx \geq 0.
	\end{eqnarray*}
	Clearly, if we let $\delta \to 0$, we have 
	\begin{eqnarray*}
		\int_{\Omega} \left(|D u_{\varepsilon}|^2 + \delta \right)^{\frac{\tp(\cdot)-2}{2}}\ip{D u_{\varepsilon}}{D \psi} dx \;\;\; \to \;\;\;
		\int_{\Omega} |D u_{\varepsilon}|^{\tp(\cdot) - 2}\ip{D u_{\varepsilon}}{D \psi} dx.
	\end{eqnarray*}
	It remains to show 
	\begin{eqnarray}\label{RTS}
		\int_{\Omega} -\div \left(\left( D u_{\varepsilon}|^2 + \delta \right)^{\frac{\tp(\cdot)-2}{2}} D u_{\varepsilon} \right) \psi dx \geq 0
	\end{eqnarray}
	as $\delta \to 0$.
	
	\begin{claim}\label{D2est}
		By the definition of $u_{\varepsilon}$,
		$$D^2 u_{\varepsilon}(x_0) \leq \frac{\tq -1}{\varepsilon} |D u_{\varepsilon}(x_0)|^{\frac{\tq -2}{\tq -1}}I_n.$$
	\end{claim}
	\begin{proof}[Proof of Claim \ref{D2est}]
	\renewcommand{\qedsymbol}{$\blacksquare$}	
		Consider a point $x_0$  and assume that $D u_{\varepsilon}(x_0)$ and $D^2 u_{\varepsilon}(x_0)$ exist.  By Item \ref{wutbB} of Lemma \ref{lscuscgradeps0}, we have 
		$|y-x_0| \leq |D u_{\varepsilon}(x_0)|^{\frac{1}{\tq -1}}\varepsilon$ for every $y \in Y_{\varepsilon}(x_0)$, where $Y_{\varepsilon}$ is defined by Equation \eqref{Yeps}.  Additionally, Item \ref{wutbA} of Lemma \ref{lscuscgradeps0} implies that for every $m$, there is a radius $R_m$ and a ball $B_{R_m}(x_0)$ such that
		\begin{eqnarray*}
			|y-x| \leq |D u_{\varepsilon}(x_0)|^{\frac{1}{\tq -1}}\varepsilon + \frac{1}{m} =: \alpha_m
		\end{eqnarray*}
		for every $x \in B_{R_m}(x_0)$ and every $y \in Y_{\varepsilon}$.  It follows that 
		\begin{eqnarray*}
			u_{\varepsilon}(x) = \inf_{y \in B_{\alpha_m}(x)} \left(u(y) + \frac{|x-y|^{\tq}}{\tq \varepsilon^{\tq -1}} \right)
		\end{eqnarray*}
		for every $x \in B_{R_m}(x_0)$.  Let 
		$$\varphi_y(x) := u(y) + \frac{|x-y|^{\tq}}{\tq \varepsilon^{\tq -1}}.$$  
		Then for every $y \in B_{\alpha_m}(x)$, $\varphi_y$ is smooth and 
		$$D^2 \varphi_y(x) \leq \frac{\tq -1}{\varepsilon^{\tq -1}} \alpha_m^{\tq -2}I_n.$$
		Because $u_{\varepsilon}$ is semi-concave and because the above holds in particular when $y \in Y_\varepsilon$, it follows that almost everywhere in $B_{\alpha_m}(x)$ we have
		$$D^2 u_{\varepsilon} \leq \frac{\tq -1}{\varepsilon^{\tq -1}} \alpha_m^{\tq -2}I_n.$$
		We let $m \to \infty$ and obtain 
		\begin{equation*}
			D^2 u_{\varepsilon}(x_0) \leq \frac{\tq -1}{\varepsilon} |D u_{\varepsilon}(x_0)|^{\frac{\tq -2}{\tq -1}}I_n.\qedhere
		\end{equation*}
	\end{proof}
	The claim above permits us to obtain the inequality
	\begin{eqnarray}\label{hessiangeq}
		D^2 u_{\varepsilon}(x) \leq 0 \;\; \textnormal{if} \;\; D u_{\varepsilon}(x)=0.
	\end{eqnarray}
	\begin{claim}\label{lapest}
		If $D u_{\varepsilon}(x) \ne 0$, then we have
		\begin{equation*}\label{estim2}
			\begin{aligned}
			     -\left(|D u_{\varepsilon}|^2 + \delta \right)^{\frac{\textnormal{\tp}(\cdot) -2}{2}} \big(\Delta u_{\varepsilon}+\frac{(\textnormal{\tp}(\cdot) -2)}{|D u_{\varepsilon}|^2+\delta}  \ip{D^2u_{\varepsilon} \;  u_{\varepsilon}}{D u_{\varepsilon}} \big)& \\
			    -\left(|D u_{\varepsilon}|^2 + \delta \right)^{\frac{\textnormal{\tp}(\cdot) -2}{2}} \log(|D u_{\varepsilon}|^2+\delta) \ip{D \textnormal{\tp}(\cdot)}{D u_{\varepsilon}}& \geq -\frac{n}{\varepsilon}  |D u_{\varepsilon}|^{\frac{\textnormal{\tp}^+ + \tq-4}{\tq -1}} \\
			    &\hspace{2.5ex}\times\bigg(\tq -1 + \textnormal{\tp}^+ -2 + C_0 \widehat{\kappa} 
			    D u_{\varepsilon}^{\frac{1}{\tq - 1}}\bigg),
			\end{aligned}
		\end{equation*}
		where $C_0 = \log(|D u_{\varepsilon}|^2+\delta)\varepsilon$ and $|D\tp(\cdot)|\leq \widehat{\kappa}$.
	\end{claim}
	\begin{proof}[Proof of Claim \ref{lapest}]
	\renewcommand{\qedsymbol}{$\blacksquare$}
		To improve the readability of our work, we make use of the notation 
		\begin{equation*}
			\begin{aligned}
				\Gamma u_\varepsilon &:=-\left(|D u_{\varepsilon}|^2 + \delta \right)^{\frac{\textnormal{\tp}(\cdot) -2}{2}} \big(\Delta u_{\varepsilon}  +  \frac{(\textnormal{\tp}(\cdot) -2)}{|D u_{\varepsilon}|^2+\delta}  \ip{D^2u_{\varepsilon} \;  Du_{\varepsilon}}{D u_{\varepsilon}} \big) \\
				&\hspace{2.5ex}-\left(|D u_{\varepsilon}|^2 + \delta \right)^{\frac{\textnormal{\tp}(\cdot) -2}{2}} \log(|D u_{\varepsilon}|^2+\delta) \ip{D \textnormal{\tp}(\cdot)}{D u_{\varepsilon}}.
			\end{aligned}
		\end{equation*}
		We have
		\begin{equation}\label{eig}
			\begin{aligned}
				\Gamma u_\varepsilon &\geq -\left(|D u_{\varepsilon}|^2 + \delta \right)^{\frac{\textnormal{\tp}^+ -2}{2}} \times \big( \frac{n(\texttt{q}-1)}{\varepsilon} |D u_{\varepsilon}|^{\frac{\tq-2}{\tq -1}} +  \frac{(\textnormal{\tp}^+ -2)}{D u_{\varepsilon}^2+\delta}  \frac{n(\texttt{q}-1)}{\varepsilon} |D u_{\varepsilon}|^{\frac{3\tq-4}{\tq -1}} \big) \\
				&\hspace{2.5ex} -\left(|D u_{\varepsilon}|^2 + \delta \right)^{\frac{\textnormal{\tp}^+ -2}{2}} \log(|D u_{\varepsilon}|^2+\delta) n \widehat{\kappa} |D u_{\varepsilon}| \\ 
				&\geq -n \frac{\tq-1}{\varepsilon} \left(|D u_{\varepsilon}|^2 + \delta \right)^{\frac{\textnormal{\tp}^+ -2}{2}} |D u_{\varepsilon}|^{\frac{\tq-2}{\tq -1}} \big(1 + \frac{(\textnormal{\tp}^+ -2)}{|D u_{\varepsilon}|^2+\delta} |D u_{\varepsilon}|^{2}
				\big) \\
				&\hspace{2.5ex} -n\widehat{\kappa} |D u_{\varepsilon}|  \left(|D u_{\varepsilon}|^2 + \delta \right)^{\frac{\textnormal{\tp}^+ -2}{2}}\log(|D u_{\varepsilon}|^2+\delta) \\ 
				&=  -\frac{n}{\varepsilon} \left(|D u_{\varepsilon}|^2 + \delta \right)^{\frac{\textnormal{\tp}^+ -2}{2}} |D u_{\varepsilon}|^{\frac{\tq-2}{\tq -1}}
				\bigg((\tq -1)+  \frac{(\textnormal{\tp}^+ -2)}{|D u_{\varepsilon}|^2+\delta} |D u_{\varepsilon}|^2 + C_0 \widehat{\kappa} |D u_{\varepsilon}|^{\frac{1}{\tq - 1}}\bigg), 
			\end{aligned}
		\end{equation}
		where 
		$$C_0 = \varepsilon \log(|D u_{\varepsilon}|^2+\delta).$$
		Continuing with the estimate, 
		\begin{equation*}
	        \begin{aligned}
	            &-n \frac{1}{\varepsilon} \left(|D u_{\varepsilon}|^2 + \delta \right)^{\frac{\textnormal{\tp}^+ -2}{2}} |D u_{\varepsilon}|^{\frac{\tq-2}{\tq -1}} \\
			    &\times\bigg((\tq -1)+ \frac{(\textnormal{\tp}^+ -2)}{|D u_{\varepsilon}|^2+\delta} |D u_{\varepsilon}|^2 + C_0 \widehat{\kappa} |D u_{\varepsilon}|^{\frac{1}{\tq - 1}}\bigg) \geq -\frac{n}{\varepsilon}  |D u_{\varepsilon}|^{\frac{\tp^+ + \tq-4}{\tq -1}} \\
			    &\hspace{51ex}\times\bigg(\tq -1 +  \tp^+ -2 + C_0 \widehat{\kappa} 
			    D u_{\varepsilon}^{\frac{1}{\tq - 1}}\bigg)
	        \end{aligned}
		\end{equation*}
		and this combined with \eqref{eig} completes the claim.
	\end{proof}
	Now, since $u_{\varepsilon}$ is Lipschitz continuous and $\tq > \frac{\tp}{\tp-1}$, then by employing Claim \ref{lapest}, we can write:
	\begin{equation}\label{simplest}
		\begin{aligned}
		    -\left(|D u_{\varepsilon}|^2 + \delta \right)^{\frac{\textnormal{\tp}(\cdot) -2}{2}} \big(\Delta u_{\varepsilon}+ \frac{(\textnormal{\tp}(\cdot) -2)}{|D u_{\varepsilon}|^2+\delta}  \ip{D^2u_{\varepsilon} \;  u_{\varepsilon}}{D u_{\varepsilon}} \big)& \\
		    -\left(|D u_{\varepsilon}|^2 + \delta \right)^{\frac{\textnormal{\tp}(\cdot) -2}{2}} \log(|D u_{\varepsilon}|^2+\delta) \ip{D \textnormal{\tp}(\cdot)}{D u_{\varepsilon}}
		    &\geq -C.
		\end{aligned}
	\end{equation}
	It follows that we can now apply Fatou's Lemma and use Equation \eqref{hessiangeq} to obtain
	\begin{eqnarray*}
		&&\liminf_{\delta \to 0} \int_{\Omega} -\div \left(\left(|D u_{\varepsilon}|^2 + \delta \right)^{\frac{\tp - 2}{2}}D u_{\varepsilon}\right) \psi dx \\
		&\geq& \int_{\Omega \setminus \{D u_{\varepsilon} = 0\}} \liminf_{\delta \to 0} \left(-\div \left(\left(|D u_{\varepsilon}|^2 + \delta \right)^{\frac{\tp(\cdot)-2}{2}} D u_{\varepsilon}\right) \psi \right) dx \\
		&=& \int_{\Omega \setminus \{D u_{\varepsilon} = 0\}} -\div \left(|D u_{\varepsilon}|^{\tp(\cdot)-2} D u_{\varepsilon}\right) \psi dx \\
		&=& \int_{\Omega \setminus \{D u_{\varepsilon} = 0\}} -\Delta_{\tp(\cdot)} u_{\varepsilon} \psi dx \\
		&\geq& 0,
	\end{eqnarray*}
	where we utilize comments preceeding Equation \eqref{NTS2} for the last inequality, giving us Equation \eqref{RTS}.  Then letting $\delta \to 0$ in Equation \eqref{NTS2} and following the techniques utilized in Subsection \ref{degcase} gives us 
		\begin{equation*}
	    \int_{\Omega} \left|Du_\varepsilon\right|^{\tp(\cdot)-2}\ip{Du_\varepsilon}{D\psi}dx \geq 0,
	\end{equation*}
completing the proof of Case $2$.
\end{proof}

\subsection[main]{The Main Result}\label{mainresult}
We have shown that viscosity supersolutions and $\tp(\cdot)$-superharmonic  solutions to the $\tp(\cdot)$-Laplace equation coincide.
An analogous argument shows viscosity subsolutions and $\tp(\cdot)$-subharmonic solutions to the $\tp(\cdot)$-Laplace equation are also the same class of solution.  
The Main Theorem follows as an immediate consequence of Theorem \ref{main2}, that analogous argument, and Corollary \ref{weakharm}.
\begin{main}
Potential theoretic weak solutions and viscosity solutions to the \\ $\textnormal{\tp}(\cdot)$-Laplace equation coincide.
\end{main}

\end{document}